\documentclass[11pt]{amsart}
\usepackage{amsmath, amsfonts, amssymb, epsfig, graphicx}

\theoremstyle{plain}

\newtheorem{theorem}{Theorem}
\newtheorem{observation}{Observation}
\newtheorem{lemma}{Lemma}

\newtheorem{conjecture}{Conjecture}

\theoremstyle{definition}

\DeclareMathOperator{\interior}{int}
\DeclareMathOperator{\relint}{rel\,int}
\DeclareMathOperator{\bd}{bd}

\DeclareMathOperator{\aff}{aff}
\DeclareMathOperator{\convex}{conv}

\newcommand{\R}{\mathbb{R}}

\newcommand{\half}{\tfrac{1}{2}}

\newcommand{\ipr}[2]{\langle #1, #2 \rangle}

\newcommand{\norm}[1]{\lVert#1\rVert}

\newcommand{\abs}[1]{\lvert#1\rvert}

\begin{document}

\bibliographystyle{amsplain}

\title{Midpoint sets contained in the unit sphere of a normed space}
\author{Konrad J.\ Swanepoel}
\address{Department of Mathematics,
London School of Economics and Political Science,
Houghton Street,
London WC2A 2AE,
United Kingdom}
\email{k.swanepoel@lse.ac.uk}
\subjclass[2000]{Primary 52A15, Secondary 52A20}
\keywords{Minkowski space, midpoint set, convex position}

\begin{abstract}
The midpoint set $M(S)$ of a set $S$ of points is the set of all midpoints of pairs of points in $S$.
We study the largest cardinality of a midpoint set $M(S)$ in a finite-dimensional normed space, such that $M(S)$ is contained in the unit sphere, and $S$ is outside the closed unit ball.
We show in three dimensions that this maximum (if it exists) is determined by the facial structure of the unit ball.
In higher dimensions no such relationship exists.
We also determine the maximum for euclidean and sup norm spaces.
\end{abstract}

\maketitle

\section{Introduction and main results}

For any set of points $S$ in a vector space, the {\em midpoint set} of $S$ is
\[ M(S) = \{ \half (x+y) : x,y\in S, x\neq y\}.\]
We study sets such that all their midpoints are unit vectors in a finite-dimensional normed space.
This work is contained in the author's PhD thesis \cite{Swanepoel1997}.
Since recently some work appeared \cite{Halman2007, Eisenbrand2008} dealing with related questions, the author decided to write this up as a paper.

First note that if the space is not strictly convex, there exists an infinite set whose midpoint set is a set of unit vectors:
Choose distinct points on a line segment on the boundary of the unit ball.
If however $X$ is strictly convex, it follows from the strict triangle inequality that there is at most one point in $S$ of norm $\leq 1$.

We therefore only look at the cases where all the vectors in $S$ are of norm larger than one.
We say that a set of vectors $S\subseteq X$ is an {\em M-set} if $M(S)$ consists of unit vectors, and each vector in $S$ is of norm $>1$.
Let $m(X)$ be the largest $m$ such that there exists an M-set of cardinality $m$ in $X$, if there is such a largest $m$.
If there is no such largest $m$, we set $m(X)=\infty$.

In two dimensions, $m(X)$ does not distinguish between different norms.

\begin{theorem}\label{helly:th2}
We have $m(X)=3$ for any two-dimensional Minkowski space $X$.
\end{theorem}

In dimensions higher than two, $m(X)=\infty$ is possible.

\begin{theorem}\label{infvectors}
If the unit ball of a Minkowski space has a non-polytopal proper face, then there exists an infinite M-set.
\end{theorem}

However, it is easily seen that an M-set is always countable.

\begin{observation}
Any M-set in a Minkowski space is countable.
\end{observation}
\begin{proof}
It is sufficient to show that a set $S$ of vectors of norm $\geq 1+\epsilon$ satisfying $\norm{x+y}=2$ for distinct $x,y\in S$, is finite (for any $\epsilon>0$).
But such a set is obviously bounded, and for any two vectors $x,y\in S$ we have
$\norm{x-y}+2 = \norm{x-y}+\norm{x+y}\geq \norm{2x}\geq 2+2\epsilon$.
\end{proof}

The following theorem provides a converse for Theorem~\ref{infvectors} in the case of $3$-dimensional spaces.
However, we suspect that there is a $4$-di\-men\-sion\-al strictly convex, smooth space admitting an infinite M-set; see Theorem~\ref{ex4d} for a weaker result.

\begin{theorem}\label{polyth}
Let the unit ball of a $3$-dimensional Minkowski space $X$ have $2$-faces, all of which are polygonal.
Then any M-set is finite.
Furthermore,
\begin{equation*}
m(X) = \sup\{n+1:\text{ The unit ball has an $n$-gonal $2$-face}\}
\end{equation*}
\end{theorem}

The only possibility for three-dimensional spaces not yet covered, is where the unit ball does not have a $2$-face.

\begin{theorem}\label{threelower}
If the unit ball of a $3$-dimensional Minkowski space $X$ does not have a $2$-face, then $m(X)=4$.
\end{theorem}

\begin{conjecture}
 There exists a strictly convex, smooth $4$-dimensional Minkowski space admitting an infinite M-set.
\end{conjecture}

\begin{theorem}\label{ex4d}
For all $n\geq 1$ there exists a $4$-dimensional smooth, strictly convex Minkowski space $X$ such that $m(X)\geq n$.
\end{theorem}

We finally calculate $m(X)$ for euclidean and $\ell_\infty$ spaces of all dimensions.

\begin{theorem}\label{l2m}
For all $d\geq 1$, $m(\ell_2^d)=d+1$.
\end{theorem}

\begin{theorem}\label{inftybound}
For all $d\geq 2$, $m(\ell_\infty^d)=2d-1$.
\end{theorem}

In the next section we esablish some notation and prove some basic properties of M-sets.
In Section~\ref{proofsection} the above theorems are proved.
In the last section we discuss some open questions.

\section{Notation and Basic properties}
We denote the convex hull, affine hull, boundary, interior, and relative interior of a set $S$ in a vector space by $\convex S, \aff S, \bd S, \interior S$ and $\relint S$, respectively.
By {\em polytope} we always mean a convex polytope.
We say that a polytope $P$ is $2$-neighbourly if for any two vertices $x$ and $y$, the segment $[x,y]:=\convex\{x,y\}$ is an edge of $P$, and we say that $P$ is $2$-almost-neighbourly if for any two vertices $x$ and $y$, the segment $[x,y]$ is contained in the boundary of $P$.

We frequently use the following fact:

\begin{lemma}\label{normface}
Let $S$ be a set of unit vectors in a Minkowski space.
If there exists a unit vector $y\in\relint\convex S$, then $\convex S$ is a set of unit vectors.
\end{lemma}
\begin{proof}
By Carath\'eodory's theorem we may assume that $S$ is finite: $S=\{x_1,\dots,x_m\}$.
For any $z\in\convex S$ we have $\norm{z}\leq 1$, since the unit ball is convex.
Alternatively, if $z=\sum_i\lambda_i x_i$ for some $\lambda\geq 0$ such that $\sum_i\lambda_i=1$, then $\norm{z}=\norm{\sum_i\lambda_i x_i}\leq \sum_i\lambda_i\norm{x_i}=1$.

We may assume $z\neq y$.
Then there exist unit vectors $a,b$ on the line $\aff\{y,z\}$ such that $y$ and $z$ are between $a$ and $b$.
Thus there exists $0<\lambda\leq 1$ such that $y=\lambda a + (1-\lambda)z$, and
$1=\norm{y}\leq\lambda\norm{a}+(1-\lambda)\norm{z}=\lambda+(1-\lambda)\norm{z}$, implying $\norm{z}\geq 1$.
\end{proof}

We need the following sharpening of Carath\'eodory's Theorem:

\begin{lemma}\label{caramod}
Let $R\subseteq\R^d$ and $y\in\relint\convex R$.
Then for any $x\in R$ there is an affinely independent set $S\subseteq R$ of at most $d+1$ points such that $x\in S$ and $y\in \relint\convex S$. 
\end{lemma}
\begin{proof}
We may assume without loss of generality that $R$ is finite, and that $\convex R$ is $d$-dimensional.
If $x=y$, there is nothing to prove.
Otherwise, let $y'$ be the projection of $y$ along the ray from $x$ through $y$ onto a facet $F$ of $\convex X$.
By Carath\'eodory's theorem, $y'$ is in the relative interior of the convex hull of at most $d$ affinely independent vertices of $F$.
Add the point $x$ to obtain the set $S$.
\end{proof}

\begin{lemma}\label{polytope}
Let $M$ be a finite M-set in a Minkowski space.
Then $M$ is the vertex set of a polytope $P$.

Suppose furthermore that for some two points $x,y\in M$, the open segment $(x,y)$ is contained in the relative interior of some $k$-face $Q$ of $P$, with $k\geq 2$ (thus $[x,y]$ is not an edge).
Then the unit ball has a $k$-face in $\aff Q$.
\end{lemma}

\begin{proof}
We first show that no point is in the convex hull of the remaining points.
Suppose that $x$ is in the convex hull of the remaining vectors.
We may suppose $x = \sum_{i=1}^n \lambda_i x_i$, with $x_i\in M, \lambda_i\geq 0$ and $\sum_{i=1}^n\lambda_i = 1$.
Then 
\begin{align*}
 \norm{2x} = & \norm{\sum_{i=1}^n \lambda_i (x+x_i)} \\
		&\leq \sum_{i=1}^n\lambda_i\norm{x+x_i} \\
		&\leq 2,
\end{align*}
and $\norm{x}\leq 1$, a contradiction.

Secondly, let $x$ and $y$ be two vertices of $P$ such that $(x,y)\subseteq\relint Q$.
Note that $x,y\in Q$.
Let $R$ be the set of vertices $z$ of $Q$ such that $[x,z]$ is an edge of $Q$.
Thus $x,y\notin R$.
Let $z\in R$.
By considering the vertex figure of $Q$ at $x$ and Lemma~\ref{caramod}, there is a simplex with vertex set $S$ satisfying $z\in S\subseteq R$, such that $y-x$ is a strictly positive linear combination of $S-x$:
$$y= x+ \sum_{s\in S}\lambda_s (s-x), \qquad \lambda_s>0.$$
Note that $S$ contains at least $2$ and at most $k$ points.
We must have $\lambda:= \sum_s\lambda_s>1$, otherwise $y$ is in the convex hull of $\{x\}\cup S$, contradicting the first part of the proof.
But now $\half (x+y)$ is in the relative interior of the polytope
$$T_z:=\convex\left\{\half(x+s), \half(y+s):s\in S\right\},$$
as can be seen by taking $\mu_s := \lambda_s/\lambda^2$ and $\nu_s := \lambda_s(1-1/\lambda)/\lambda$; then
$$\half (x+y) = \sum_{s\in S}\mu_s \half(x+s) + \sum_{s\in S}\nu_s \half(y+s),\quad\sum_{s\in S}(\lambda_s+\nu_s)=1.$$
It follows that $T_z$ is in the boundary of the unit ball by Lemma~\ref{normface}.
However, it is possible that $T_z$ is not $k$-dimensional, although it is contained in $\aff Q$.

We also have $\half(x+y)\in\relint R'$, where 
$$R':=\convex\left\{\half(x+z), \half(y+z):z\in R\right\}$$
is a $k$-polytope contained in $\aff Q$:
if not, then $\half(x+y)$ is contained in some facet $F$ of $R'$.
But choose a $z\in R$ such that $\half(x+z)$ or $\half(y+z)$ is not in $F$.
Then, since $\half(x+y)\in\relint T_z$, we obtain a contradiction.
\end{proof}

\section{Proofs}\label{proofsection}

\begin{proof}[Proof of Theorem~\ref{helly:th2}]
By Lemma~\ref{polytope} we immediately get $m(X)\leq 3$.
For the upper bound, find unit vectors $x,y,z$ such that $x+y+z=0$, by inscribing an affinely regular hexagon in the unit disc.
Then $\{2x,2y,2z\}$ is an M-set.
\end{proof}

\begin{proof}[Proof of Theorem~\ref{infvectors}]
Note that the boundary of a non-polygonal $2$-face $F$ contains either infinitely many straight line segments or a strictly convex arc.
In both cases it is easy to see how to choose infinitely many points in the plane of $F$, but outside $F$, such that the midpoint of any two of the points is in $F$; see Figure~\ref{fig4}.
\begin{figure}
\includegraphics[bb = 0 0 91 59]{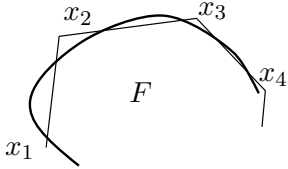}
\begin{picture}(0,0)(91,0)%
\put(35,25){$F$}
\put(-12,5){$x_1$}
\put(10,56){$x_2$}
\put(61,58){$x_3$}
\put(84,33){$x_4$}
\end{picture}%
\caption{Choosing points outside $F$ with midpoint inside $F$}\label{fig4}
\end{figure}
For faces of larger dimension, take an appropriate $2$-dimensional cut to obtain a non-polygonal convex disc $F$.
\end{proof}

\begin{lemma}\label{threelow}
For any Minkowski space of dimension at least three, $m(X)\geq 4$.
\end{lemma}
\begin{proof}
We may assume without loss of generality that the space is three-dimensional.
Consider any plane $H$ through $0$.
Let $x$ be a unit vector such that $x+H$ supports the unit ball at $x$.
For any $\lambda\in\R$, let $H_\lambda = H+\lambda x$.
Now fix $1/3 < \lambda < 1$.
The intersection of the boundary of the unit ball with $H_\lambda$ is a convex curve $C$.
It is possible to inscribe an affinely regular hexagon in $C$ (see e.g.\ \cite[p.\ 242]{Gr3}) with centre $y\in H_\lambda$.
Let the vertices be $\{y\pm a, y\pm b, y\pm c\}$, where $a,b,c\in H$ satisfy $a+b+c=0$.
Then it is easily seen that $\{y+2a, y+2b, y+2c, -3y\}$ is an M-set.
\end{proof}

\begin{proof}[Proof of Theorem~\ref{threelower}]
Let $S\subseteq X$ be an M-set.
By Lemma~\ref{polytope}, $S$ is the vertex set of a polytope of dimension $2$ or $3$, and $2$-neighbourly, implying $m\leq 4$, or of dimension $1$, implying $m\leq 2$.

The lower bound follows from Lemma~\ref{threelow}.
\end{proof}

\begin{lemma}\label{polylower}
If the unit ball of a $3$-dimensional Minkowski space $X$ has an $n$-gonal $2$-face, then $m(X)\geq n+1$.
\end{lemma}
\begin{proof}
Given the polygon $P:= \convex\{a_1,a_2,\dots, a_n\}$ (with vertices in this order) as a face of the unit ball, we construct $n+1$ points outside the unit ball such that the midpoint of any two is in $P$ or $-P$.

As an intermediate step, let $x_i:=\half(a_i+a_{i+1})$ for $i=1,\dots,n$.
Then the midpoint of any two $x_i$'s is obviously in $\relint P$.
We now want to find a point $x_{n+1}$ such that the midpoint of $x_i$ and $x_{n+1}$ is in $\relint -P$ for all $i=1,\dots,n$.
(We need to have the midpoints in the {\em interior} of the reflected polygon in order to later modify $x_i,i=1,\dots,n$ --- at the moment they are still on the boundary of the unit ball.)
See Figure~\ref{proj}.
\begin{figure}
\includegraphics[bb = 0 0 109 150]{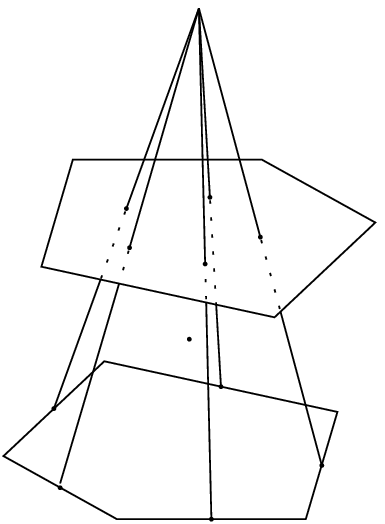}
\begin{picture}(0,0)(109,0)%
\put(23,137){$x_{n+1}$}
\put(40,48){$0$}
\put(-17,17){$P$}
\put(-22,82){$-P$}
\end{picture}%
\caption{}\label{proj}
\end{figure}
Note that such an $x_{n+1}$ would be outside the unit ball (in fact $\norm{x_{n+1}}=3$).
For any $x$, the set $\{\half(x+x_i):i=1,\dots,n\}$ is homothetic to $\{x_1,\dots,x_n\}$, with ratio $\half$.
It is therefore sufficient to show that there is a $-\half$-copy of $\{x_1,\dots,x_n\}$ in $\relint P$.
Let $P_i$ be the homothet of $\relint P$ with centre $x_i$ and factor $\tfrac{2}{3}$.
It is easily seen that $c\in P_i$ iff $c$ is a centre of homothety with factor $-\half$ taking $x_i$ into $\relint P$.
We want to find such a $c$ working for all $i$; it is therefore sufficient to show that $\bigcap_{i=1}^n P_i\neq\emptyset$.
By Helly's theorem it is sufficient to show that any three $P_i$'s have non-empty intersection.
This is equivalent to showing that there is a $-\half$-copy of $\{x_1,\dots,x_n\}$ in $\relint P$ in the case $n=3$.
But there is already a $-\half$-copy of $\{x_1,x_2,x_3\}$ in $(\convex\{x_1,x_2,x_3\})\setminus\{x_1,x_2,x_3\}$: for $c$ take the centroid of $\{x_1,x_2,x_3\}$.
(Note that $(\convex\{x_1,x_2,x_3\})\setminus\{x_1,x_2,x_3\}\subseteq \relint P$.)

We now modify $x_1,\dots,x_n$ by slightly shifting each point in the plane $\aff P$ out of $P$ so that $\norm{x_i}>1$.
As the construction was made such that all the midpoints were either in $\relint P$ or $\relint -P$, the midpoints will now still be in $P$ or $-P$.
\end{proof}

The following lemma was first proved by Zamfirescu \cite{Zamfirescu1969}.
For convenience we include a proof.

\begin{lemma}\label{threepolytope}
Let $P$ be a $2$-almost-neighbourly $3$-polytope.
Then $P$ is either combinatorially equivalent to a triangular prism, or $P$ is a polygonal pyramid. \textup{(}See Figure ~\ref{fig1}.\textup{)}
\end{lemma}

\begin{figure}
\includegraphics[bb = 0 0 299 124]{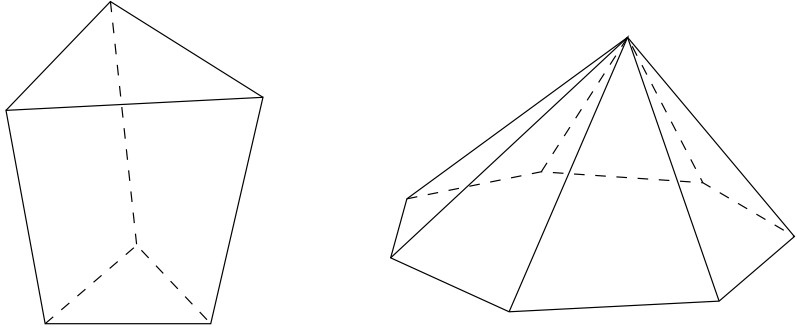}
\caption{Combinatorial triangular prism. Polygonal pyramid.}\label{fig1}
\end{figure}

\begin{proof}
Suppose that two faces, an $n$-gon $F$ and an $n'$-gon $G$ , are joined by an edge, with $n,n'\geq 4$.
\begin{figure}
\includegraphics[bb = 0 0 166 106]{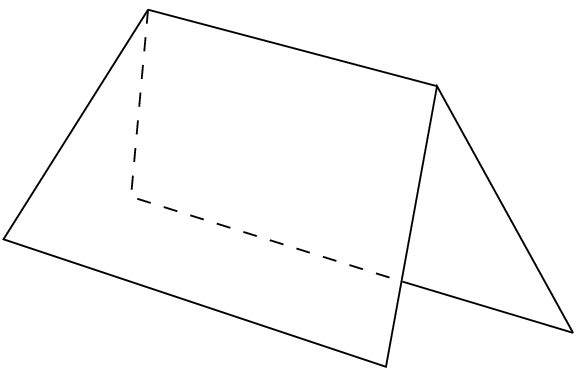}
\begin{picture}(0,0)(166,7)%
\put(25,103){$a$}
\put(130,80){$b$}
\put(163,16){$c$}
\put(23,53){$d$}
\put(112,5){$e$}
\put(-14,28){$f$}
\put(145,50){$G$}
\put(40,11){$F$}
\end{picture}%
\caption{}\label{fig2}
\end{figure}
Let the vertices of the joining edge be $a$ and $b$. Let the vertices of $F$ in the neighbourhood of $a$ and $b$ be $f,a,b,e$, and those of $G$ be $d,a,b,c$.
See Figure~\ref{fig2}.
Then $[d,e]$ and $[c,f]$ are on the boundary of $P$ and must be contained in the same face of $P$.
Thus, $\convex\{d,c,e,f\}$ is part of this face.
It follows that $n=n'=4$.
If $\convex\{b,c,e\}$ is not a face of $P$, then $\aff\{b,c,e\}$ strictly separates $a$ from some vertex $x$ of $P$.
Since $x$ cannot be part of $F$ or $G$, it follows that $[a,x]\nsubseteq\bd P$, a contradiction.
Thus, $\convex\{b,c,e\}$ is a face of $P$, and similarly, $\convex\{a,d,f\}$ is a face.
Thus $P$ is a triangular prism.

If $P$ is not combinatorially equivalent to a triangular prism it follows that if two faces meet at an edge, one must be a triangle. 
Let $x_1x_2\dots x_n$ be a face with the largest number of vertices $n$.
Then all faces adjacent to $F$ are triangles.
Suppose there are two points $y_1,y_2$ not in the plane $\aff F$.
By Radon's theorem applied to $S:=\{x_1,x_2,x_3,y_1,y_2\}$, there is a partition of $S$ into sets $T,U$ such that $\convex T\cap\convex U\neq\emptyset$.
Since $S$ is the vertex set of a polytope, $T$ and $U$ both contain at least two points.
We may assume that $\#T=2$ points, and $\#U=3$.
It is clear that $T=\{x_i,y_j\}$ for some $i,j$, say $T=\{x_1,y_1\}$.
Thus $[x_1,y_1]$ intersects $\convex U$.
Since $[x_1,y_1]$ is on the boundary of $\convex S$, it cannot intersect $\relint\convex U$.
Therefore, it must intersect the relative interior of some edge of $\convex U$, and we obtain a face of at least $4$ vertices adjacent to $F$, a contradiction.
Thus there is exactly one point not in $\aff F$, and $P$ is a pyramid.
\end{proof}

\begin{proof}[Proof of Theorem~\ref{polyth}]
The lower bound in the equation is proved in Lemma~\ref{polylower}.
By Lemmas~\ref{polytope} and \ref{threepolytope}, any finite M-set is the vertex set of a planar polygon, a combinatorial triangular prism, or a polygonal pyramid.

It is easily seen that a triangular prism is impossible:
Considering the midpoints of a quadrilateral face $F$, we see that the unit ball has a face in the plane of $F$.
Since this holds for all three quadrilateral faces, we contradict the central symmetry of the unit ball.

We may assume that $m\geq 5$, since otherwise
$$m(X)\leq 4\leq \sup\{n+1: \text{there is an $n$-gon}\}.$$
The only remaining possibility is that the convex hull of the M-set is a polygon of $\geq 5$ vertices or a polygonal pyramid with $\geq 4$ vertices.
In both cases, there is an M-set with a polygon $P$ as convex hull.
Thus, by Lemma~\ref{polytope} there is a $2$-face $Q$ on the boundary of the unit ball in the plane of $P$.
By hypothesis, $Q$ is a polygon.
Choose any point in $\relint Q$ and triangulate $\aff Q$ into unbounded regions by joining this point to the vertices of $Q$, as in Figure~\ref{poly}.
\begin{figure}
\includegraphics[bb = 0 0 95 89]{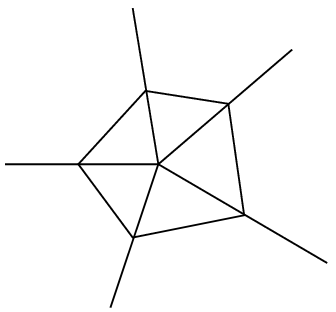}
\caption{}\label{poly}
\end{figure}
Obviously, no two points of the M-set can be in the same region.
Thus, $P$ has at most $n$ vertices, and $m\leq n+1$.

It remains to show that an infinite M-set is impossible.
If there is such a set, then, as before, any finite subset must be a polygon or a polygonal pyramid.
It is easily seen that the polygons must all lie in the same plane.
But in this plane the unit ball has some $n$-gon as a $2$-face, and therefore the number of points in this plane is bounded above by $n$, a contradiction.
\end{proof}

\begin{lemma}\label{smoothing}
Let $S$ be the vertex set of a centrally symmetric polytope in $\R^d$.
Then there is a smooth, strictly convex norm $\norm{\cdot}$ on $\R^d$ 
such that $\norm{x}=1$ for all $x\in S$.\qed
\end{lemma}

\begin{proof}[Proof of Theorem~\ref{ex4d}]
We use Carath\'eodory's moment curve:
For $t\in\R$, define $\phi(t):=(\cos t,\sin t,\cos 2t,\sin 2t)\in\R^4$.
Choose distinct $t_1,\dots,t_n\in[0,\pi/4]$, and let $x_i:=\phi(t_i)$.
The M-set will be $\{x_1,\dots,x_n\}$, which is incidentally the vertex set of a cyclic polytope, which is $2$-neighbourly; see \cite{Gale}.
Let $K:=\convex\{\pm(x_i+x_j):1\leq i<j\leq n\}$.
We now show that for each $i\neq j$, $\pm(x_i+x_j)$ is a vertex of $K$.

Define $p_{ij}(t):=\bigl(1-\cos(t-t_i)\bigr)\bigl(1-\cos(t-t_j)\bigr)$.
Then $p_{ij}(t)\geq 0$ with equality iff $t=t_i$ or $t=t_j$, and also $p_{ij}(t)<1/10$ for all $t\in[0,\pi/4]$.
By expanding $p_{ij}$, it is found that $p_{ij}(t) = \ipr{y_{ij}}{\phi(t)}+c_{ij}$, where
$$y_{ij}:= \left(-\cos t_i-\cos t_j,-\sin t_i-\sin t_j,\half\cos(t_i+t_j),\half\sin(t_i+t_j)\right)$$
and $c_{ij}:= 1+\half\cos(t_i-t_j)$.

We now show that $\{x\in\R^4:\ipr{y_{ij}}{x}+2c_{ij}=0\}$ is a support hyperplane of $K$ containing only $x_i+x_j$.
This follows from the following calculations:
\begin{align*}
\ipr{y_{ij}}{x_i+x_j}+2c_{ij} &= p_{ij}(t_i)+p_{ij}(t_j) = 0,\\
\ipr{y_{ij}}{-x_i-x_j}+2c_{ij} &= -p_{ij}(t_i)-p_{ij}(t_j)+4c_{ij}\\
&= 4c_{ij} > 0,\\
\ipr{y_{ij}}{x_i+x_k}+2c_{ij} &= p_{ij}(t_i)+p_{ij}(t_k)=p_{ij}(t_k)>0,\\
\ipr{y_{ij}}{-x_i-x_k}+2c_{ij} &= -p_{ij}(t_i)-p_{ij}(t_k)+4c_{ij}\\
&= -p_{ij}(t_k)+4+2\cos(t_i+t_j)\\
&\geq -\tfrac{1}{10}+4+\sqrt{2} > 0,\\
\ipr{y_{ij}}{x_k+x_l}+2c_{ij} &= p_{ij}(t_k)+p_{ij}(t_l)>0,\\
\ipr{y_{ij}}{-x_k-x_l}+2c_{ij} &= p_{ij}(t_k)-p_{ij}(t_l)+4c_{ij}\\
&\geq -\tfrac{1}{10}-\tfrac{1}{10}+4+\sqrt{2} > 0,
\end{align*}
for distinct $i,j,k,l$.
By Lemma~\ref{smoothing} there is a strictly convex, smooth norm on $\R^4$ such that $\norm{x_i+x_j}=2$ for all distinct $i,j$.
By strict convexity, at most one of the $x_i$'s can have norm $\leq 1$.
Note that $\dim K=4$.
This follows from the fact that any $5$ of the $x_i$'s are affinely independent, since $\{x_1,\dots,x_n\}$ is the vertex set of a polytope, as noted previously.
\end{proof}

\begin{theorem}\label{l2bound}
Let $x_1,\dots,x_m\in\ell_2^d$ be distinct vectors satisfying 
\begin{equation}\label{normeq}
\norm{x_i+x_j}=1 \text{ for all distinct }i,j.
\end{equation}
Then $m\leq d+1$ if $d\neq 2$, and $m\leq 4$ if $d=2$.
Both bounds are sharp.
\end{theorem}
\begin{proof}
Note that by the strict triangle inequality, there is at most one $x_i$ of norm $\leq\half$.
Thus for $d=2$, Theorem~\ref{helly:th2} takes care of the upper bound, and for the lower bound, take $x_1,x_2,x_3$ to be unit vectors at $120^\circ$ angles, and $x_4=0$.
Then \eqref{normeq} is obviously satisfied.

We now assume that $d\neq 2$, and for the sake of contradiction that $m=d+2$.
Then $x_1,\dots,x_m$ are affinely dependent, and there exist $\lambda_1,\dots,\lambda_m\in\R$ such that
$$\sum_{i=1}^m\lambda_i x_i = 0, \sum_{i=1}^m\lambda_i = 0, \quad(\lambda_1,\dots,\lambda_m)\neq 0.$$
From \eqref{normeq} it follows that $2\ipr{x_i}{x_j} = 1-\ipr{x_i}{x_i} - \ipr{x_j}{x_j}$ for all $i\neq j$.
Therefore, for any $j$,
\begin{align*}
0 &= \sum_{i=1}^m\lambda_i\ipr{x_i}{x_j} \\
  &= 2\lambda_j\ipr{x_j}{x_j} + \sum_{i\neq j}\lambda_i\left(1-\ipr{x_i}{x_i} - \ipr{x_j}{x_j}\right) \\
  &= 2\lambda_j\ipr{x_j}{x_j} - \lambda_j\left(1-2\ipr{x_j}{x_j}\right) - \sum_{i=1}^m\lambda_i\ipr{x_i}{x_i} \\
  &= \lambda_j\left(4\ipr{x_j}{x_j} - 1\right) - \sum_{i=1}^m\lambda_i\ipr{x_i}{x_i}.
\end{align*}
Thus
\begin{equation}\label{lambdaj}
\lambda_j(4\ipr{x_j}{x_j} -1) = \sum_{i=1}^m\lambda_i\ipr{x_i}{x_i}.
\end{equation}
Summing over all $j$, we obtain $$(m-4)\sum_{i=1}^m\lambda_i\ipr{x_i}{x_i}=0.$$
Since $d\neq 2$, we have $m\neq 4$, and by \eqref{lambdaj} we obtain that $\lambda_j=0$ for all $j$ such that $\norm{x_j}\neq\half$.
Since there is at most one $x_j$ of norm $\half$, and $\sum_i\lambda_i = 0$, we obtain that $\lambda_j=0$ for all $j$, contradicting affine dependence.

To obtain $d+1$ vectors satisfying \eqref{normeq}, consider the vertices of a regular simplex with centroid at the origin.
By symmetry, the sum of any two vectors will have the same norm.
By a suitable scaling, \eqref{normeq} will be satisfied.
\end{proof}

Theorem~\ref{l2m} now follows immediately from Theorems~\ref{l2bound} and \ref{helly:th2}.

\begin{proof}[Proof of Theorem~\ref{inftybound}]
Let $\{x_1,\dots,x_m\}$ be an M-set in $\ell_\infty^d$.
For each $i$ there is a coordinate of $x_i$ of absolute value larger than $1$.
If $m\geq 2d+1$, then by the pigeon-hole principle, there are three $x_i$'s with the same coordinate of absolute value $>1$.
Some two of these will have the same sign, and therefore the norm of their sum will be $>2$, a contradiction.

Therefore, $m\leq 2d$.
Suppose now that equality holds.
Then, by the same considerations as above, we may reorder the $x_i$'s such that $x_i^{(i)}>1$ and $x_{i+d}^{(i)}<-1$ for $i=1,\dots,d$, and $x_i^{(j)}\leq 1$ in all other cases.
If $x_i^{(i)}+x_{i+d}^{(i)}\geq 2$ for some $i$, then $x_1^{(i)}\geq 2-x_{i+d}^{(i)}>3$.
Thus $x_j^{(i)}<-1$ for all $j\neq i$, since $\norm{x_i+x_j}_\infty =2$.
But then $\norm{x_j+x_{j'}}_\infty >2$ for any two $j,j'\neq i$, a contradiction.
Similarly, $x_i^{(i)}+x_{i+d}^{(i)}\leq -2$ is impossible.
Since $\norm{x_i+x_{i+d}}_\infty=2$, there is a coordinate $j\neq i$ such that $x_i^{(j)} + x_{i+d}^{(j)} = \pm 2$.
Since $\abs{x_i^{(j)}},\abs{x_{i+d}^{(j)}}\leq 1$, we must have $x_i^{(j)}=x_{i+d}^{(j)}=\pm 1$.
But then either $\norm{x_i+x_j}_\infty>2$ or $\norm{x_{i+d}+x_j}_\infty>2$, a contradiction.
Thus $m\leq 2d-1$.

For the lower bound, let $e_1,\dots,e_d\in\R^d$ be the standard unit vectors.
For $i=1,\dots,d-1$, let $x_i:=2e_i+ e_d$ and $x_{i+d-1}:=-2e_1+ e_d$. Finally, let $x_{2d-1}:=-3 e_d$.
\end{proof}

\section{Concluding remarks}

Any centrally symmetric polytope $P$ defines a Minkowski space $X$ for which $m(P)=m(X)$ is finite.
It would be of some interest to study this number for specific polytopes.
In Theorem~\ref{inftybound} we have already determined $M(P)$ for cubes.

We make the following simple observations.

\begin{observation}\label{polytopeupper}
If a Minkowski space $X$ of dimension at least $2$ has a polytope with $f$ facets as unit ball, then $m(X)\leq f-1$.
\end{observation}
\begin{proof}
Imbed the space isometrically into $\ell_\infty^{f/2}$ and apply Theorem~\ref{inftybound}.
Note that since the dimension is at least $2$, $f\geq 4$.
\end{proof}

\begin{observation}\label{polytopelower1}
Let $X$ be a Minkowski space with a polytopal unit ball.
Let $f$ be the largest number of facets that a proper face of the unit ball can have.
Then $m(X)\geq f$.
\end{observation}
\begin{proof}
Let $F$ be a face of the unit ball that has $f$ facets.
Choose points $x_1,\dots,x_f$ just outside the centroid of each facet of $F$ in $\aff F$.
Then $\{x_1,\dots,x_f\}$ is clearly an M-set.
\end{proof}

\begin{observation}\label{polytopelower2}
In any $d$-dimensional Minkowski space with a polytope as unit ball, there exists an M-set of cardinality at least $d+1$.
\end{observation}
\begin{proof}
Choose any facet $F$ of the unit ball.
Let $F$ have $f$ facets.
By the argument of Theorem~\ref{polytopelower1}, $m(X)\geq f$.
If $f\geq d+1$, we are done.
Otherwise $F$ is a simplex with centroid $c$, say.
It is easily seen that we may add $-3c$ to the M-set obtained in Theorem~\ref{polytopelower1}, to obtain an M-set of $d+1$ points.
\end{proof}

These estimates are not sharp.
As a start, consider cross-polytopes.
By Observation~\ref{polytopelower2} we obtain $m(\ell_1^d)\geq d+1$, and Observation~\ref{polytopelower1} doesn't give any larger lower bound.
As upper bound we only have $2^d-1$, by Observation~\ref{polytopeupper}.
Perhaps there is also a linear upper bound?

\begin{conjecture}
 Each $d$-dimensional Minkowski space has an M-set of at least $d+1$ points.
\end{conjecture}
Theorem~\ref{helly:th2} and Lemma~\ref{threelow} show that this conjecture is true for $d=2$ and $3$.
Showing it for $d\geq 4$ would undoubtedly be more difficult than Observation~\ref{polytopelower2}.

\section*{Acknowledgements}
The author thanks Valeriu Soltan for drawing his attention to the work of Zamfirescu, and the anonymous referee for his suggestions on improving the paper.

\end{document}